\documentclass[11pt,a4paper]{article}

\usepackage{inputenc}
\usepackage{amsmath}
\usepackage{bm}
\usepackage{bbold}
\usepackage{amsthm}

\usepackage{hyperref}

\title{A New Algebraic Solution\\ to Multidimensional Minimax Location Problems with Chebyshev Distance\thanks{WSEAS Transactions on Mathematics, 2012. Vol.~11, no.~7, pp.~605-614.}
}

\author{Nikolai Krivulin\thanks{Faculty of Mathematics and Mechanics, St.~Petersburg State University, 28 Universitetsky Ave., St.~Petersburg, 198504, Russia, 
nkk@math.spbu.ru}
}

\date{}

\newtheorem{theorem}{Theorem}
\newtheorem{lemma}{Lemma}
\newtheorem{corollary}{Corollary}

\begin{document}

\maketitle

\begin{abstract}
Both unconstrained and constrained minimax single facility location problems are considered in multidimensional space with Chebyshev distance. A new solution approach is proposed within the framework of idempotent algebra to reduce the problems to solving linear vector equations and minimizing functionals defined on some idempotent semimodule. The approach offers a solution in a closed form that actually involves performing matrix-vector multiplications in terms of idempotent algebra for appropriate matrices and vectors. To illustrate the solution procedures, numerical and graphical examples of two-dimensional problems are given.
\\

\end{abstract}

\section{Introduction}

In the area of optimization, location problems  \cite{Eiselt11Pioneering} constitute a research domain of continuing interest that goes back to the seventeenth century and classical works of P.~Fermat, E.~Torricelli, J.~J.~Sylvester, J.~Steiner, and A.~Weber. Many results in the domain are recognized as important contributions to various research fields including integer programming, combinatorial and graph optimization \cite{Eiselt11Pioneering,Elzinga72Geometrical,Hansen81Constrained,Sule01Logistics,Moradi09Single,Drezner11Continuous}.

Models and methods of idempotent algebra, i.e., the linear algebra over semirings with idempotent addition, are among the approaches developed to attack location problems. Besides that, new applications of idempotent algebra to real-life problems in engineering, manufacturing, information technology, and other fields are arising and expanding \cite{Baccelli92Synchronization,Cuninghamegreen94Minimax,Kolokoltsov97Idempotent,Litvinov98Correspondence,Golan03Semirings,Heidergott05Maxplus,Butkovic10Maxlinear}. In terms of idempotent algebra, a range of problems that are nonlinear in the ordinary sense, become linear, and so more tractable for analysis and solution. 

Idempotent algebra based approaches prove to be useful for solving certain optimization problems, including idempotent analogues of linear programming problems and their extensions \cite{Butkovic10Maxlinear,Gaubert12Tropical}, as well as some location problems \cite{Cuninghamegreen94Minimax,Zimmermann03Disjunctive,Tharwat08Oneclass}. A single facility location problem on a graph is examined in \cite{Cuninghamegreen94Minimax}, where it is represented as a problem of minimizing a rational function of one variable in the idempotent algebra sense. However, the proposed solution of one-dimensional problems does not seem to be applicable in the multidimensional case.

In \cite{Zimmermann03Disjunctive,Tharwat08Oneclass}, a multidimensional constrained location problem on a graph is considered which has an objective function that takes the form of a maximum of functions each depending only on one variable, and so is called max-separable. Though an efficient technique to solve the problem is developed, this technique seems to be not suitable for problems with objective functions other than max-separable ones.   

An algebraic approach to multidimensional minimax single facility location problems with Chebyshev and rectilinear distances is proposed in \cite{Krivulin11Algebraic,Krivulin11Analgebraic,Krivulin11Anextremal,Krivulin11Algebraicsolution}. The approach is based on new results in the spectral theory in idempotent algebra, including extremal properties of eigenvalues for irreducible matrices. The solution reduces to minimization of functionals defined on idempotent semimodules and involves evaluation of the eigenvalue and eigenvectors of a matrix. 

In this paper, we present a new algebraic solution to both unconstrained and constrained minimax location problems with Chebyshev distance. The proposed approach mainly exploits methods and techniques of solving linear vector equations developed in \cite{Krivulin05Onsolution,Krivulin06Solution,Krivulin09Methods} rather than results from spectral theory. The approach offers a closed-form solution that involves performing a few usual matrix-vector operations. 

The rest of the paper is organized as follows. We begin with a brief overview of some concepts, definitions and notations of idempotent algebra, including idempotent semifields and semimodules, algebra of matrices, and linear vector equations. As another prerequisite, general solutions of linear vector equations are also outlined. Furthermore, solutions to some optimization problems are obtained to provide the basis for the analysis of location problems below.
 
We consider an unconstrained location problem and present a new solution, which appear to be more direct and somewhat simpler than that in \cite{Krivulin11Algebraic,Krivulin11Analgebraic,Krivulin11Anextremal,Krivulin11Algebraicsolution}. New constrained location problems are then examined and explicit solutions to the problems are given. Finally, numerical examples of solving both unconstrained and constrained problem in two-dimensional space are shown to illustrate the results.

\section{Definitions and Notations}

We start with a brief introduction to idempotent algebra so as to outline basic concepts, definitions and notations that together underlie results presented in the  paper. Further related details can be found in \cite{Baccelli92Synchronization,Kolokoltsov97Idempotent,Litvinov98Correspondence,Cuninghamegreen94Minimax,Golan03Semirings,Heidergott05Maxplus,Butkovic10Maxlinear}.

\subsection{Idempotent Semifield}

Consider a set $\mathbb{X}$ that is closed under addition $\oplus$ and multiplication $\otimes$ and has zero $\mathbb{0}$ and identity $\mathbb{1}$. We assume $(\mathbb{X},\mathbb{0},\mathbb{1},\oplus,\otimes)$ to be a commutative semiring, where addition is idempotent and multiplication is invertible. Since the nonzero elements in $\mathbb{X}$ form a group under multiplication, the semiring is usually referred to as the idempotent semifield. 

The power notation is introduced in the ordinary way. We define $\mathbb{X}_{+}=\mathbb{X}\setminus\{\mathbb{0}\}$. For any $x\in\mathbb{X}_{+}$ and any integer $p>0$, we have $x^{0}=\mathbb{1}$, $\mathbb{0}^{p}=\mathbb{0}$, and
$$
x^{p}=x^{p-1}\otimes x=x\otimes x^{p-1},
\qquad
x^{-p}=(x^{-1})^{p}.
$$

It is assumed that in the semifield, the integer power can be extended to the case of rational exponents, and so the semifield is taken to be radicable.

From here on, as it is customary in ordinary algebra, we drop the multiplication sign $\otimes$. The power notation is meant in the sense of idempotent algebra.

The idempotent addition induces a partial order $\leq$ such that $x\leq y$ if and only if $x\oplus y=y$. From this definition it follows that
$$
x\leq x\oplus y,
\qquad
y\leq x\oplus y,
$$
and that both addition and multiplication are isotonic.

Furthermore, we suppose that it is possible to complete the partial order into a linear order and consider the semifield as totally ordered. In what follows, the relation signs and the symbols $\min$ and $\max$ are thought of as referring to this linear order.

As an example, consider the idempotent semifield of real numbers
$$
\mathbb{R}_{\max,+}
=
(\mathbb{R}\cup\{-\infty\},-\infty,0,\max,+).
$$

In the semifield $\mathbb{R}_{\max,+}$, its null and identity elements are defined as $\mathbb{0}=-\infty$ and $\mathbb{1}=0$. For each $x\in\mathbb{R}$, there exists its inverse $x^{-1}$ equal to $-x$ in conventional arithmetic. For any $x,y\in\mathbb{R}$, the power $x^{y}$ corresponds to the arithmetic product $xy$. The partial order induced by the idempotent addition coincides with the natural linear order on $\mathbb{R}$.

\subsection{Idempotent Semimodule}

Vector operations are routinely introduced based on the scalar addition and multiplication defined on $\mathbb{X}$. Consider the Cartesian power $\mathbb{X}^{n}$ with its elements represented as column vectors. For any two vectors $\bm{x}=(x_{i})$ and $\bm{y}=(y_{i})$, and a scalar $c\in\mathbb{X}$, vector addition and multiplication by scalars follow the rules
$$
\{\bm{x}\oplus\bm{y}\}_{i}
=
x_{i}\oplus y_{i},
\qquad
\{c\bm{x}\}_{i}
=
cx_{i}.
$$ 

Fig.~\ref{F-VASM} illustrates the operations in $\mathbb{R}_{\max,+}^{2}$. Addition of two vectors on the plane is subject to a ``rectangle rule'' that defines the sum as a diagonal of a rectangle formed by coordinate axes together with horizontal and vertical lines drawn through the end points of the vectors. Scalar multiplication is equivalent to a shift along a line directed at $45^{\circ}$ angle to the axes. 
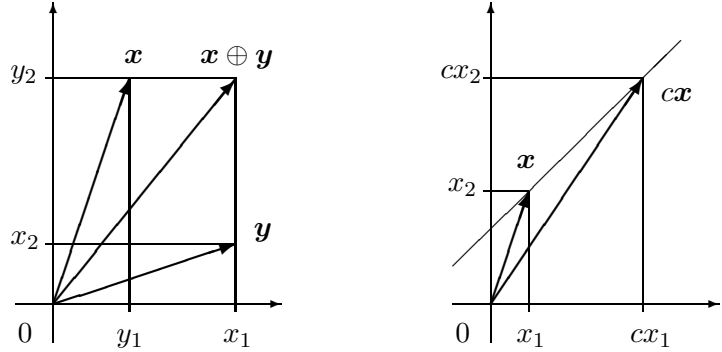
\begin{figure}[ht]
\setlength{\unitlength}{1mm}
\begin{center}
\begin{picture}(35,45)

\put(0,5){\vector(1,0){35}}
\put(5,0){\vector(0,1){45}}

\put(5,5){\thicklines\vector(1,3){10}}
%\multiput(15,35)(0,-2.9){11}{\line(0,-1){2.0}}
\put(15,35){\line(0,-1){31}}

\put(5,5){\thicklines\vector(3,1){24}}
%\multiput(29,13)(-2.9,0){9}{\line(-1,0){1.9}}
\put(29,13){\line(-1,0){25}}

\put(5,5){\thicklines\line(4,5){24}}
\put(26,31.75){\thicklines\vector(1,1){3}}

%\multiput(29,35)(-2.9,0){9}{\line(-1,0){1.9}}
\put(29,35){\line(-1,0){25}}

%\multiput(29,35)(0,-2.9){11}{\line(0,-1){2.0}}
\put(29,35){\line(0,-1){31}}

\put(0,0){$0$}

\put(13,0){$y_{1}$}
\put(27,0){$x_{1}$}

\put(-1,13){$x_{2}$}
\put(-1,35){$y_{2}$}

\put(14,37){$\bm{x}$}

\put(31,14){$\bm{y}$}

\put(24,37){$\bm{x}\oplus\bm{y}$}

\end{picture}
\hspace{20\unitlength}
\begin{picture}(35,45)

\put(0,5){\vector(1,0){35}}
\put(5,0){\vector(0,1){45}}

\put(5,5){\thicklines\vector(1,3){5}}
\put(10,20){\line(0,-1){16}}
\put(10,20){\line(-1,0){6}}

\put(5,5){\thicklines\vector(2,3){20}}
\put(25,35){\line(0,-1){31}}
\put(25,35){\line(-1,0){21}}

\put(0,10){\line(1,1){30}}

\put(0,0){$0$}
\put(8,23){$\bm{x}$}
\put(27,32){$c\bm{x}$}

\put(-1,20){$x_{2}$}
\put(-2,35){$cx_{2}$}

\put(8,0){$x_{1}$}
\put(23,0){$cx_{1}$}

\end{picture}
\end{center}
%\vspace{-2ex}
\caption{Vector addition (left) and scalar multiplication (right) in $\mathbb{R}_{\max,+}^{2}$.}\label{F-VASM}
%\vspace{-1ex}
\end{figure}

The set $\mathbb{X}^{n}$ with these operations is a semimodule over the idempotent semifield $\mathbb{X}$.

Both vector addition and scalar multiplication are isotone operations in every arguments. 

A vector with all zero entries is referred to as the zero or null vector and denoted by $\mathbb{0}$.

A vector is called regular if it has no zero entries. The set of all regular vectors of order $n$ over $\mathbb{X}$ is $\mathbb{X}_{+}^{n}$.

A vector $\bm{y}\in\mathbb{X}^{n}$ is linearly dependent on vectors $\bm{x}_{1},\ldots,\bm{x}_{m}\in\mathbb{X}^{n}$, if there are scalars $c_{1},\ldots,c_{m}\in\mathbb{X}$ such that
$$
\bm{y}
=
c_{1}\bm{x}_{1}\oplus\cdots\oplus c_{m}\bm{x}_{m}.
$$

A vector $\bm{y}$ is collinear with $\bm{x}$, if $\bm{y}=c\bm{x}$.

For any vectors $\bm{x}_{1},\ldots,\bm{x}_{m}\in\mathbb{X}^{n}$, their linear span is defined as the set
$$
\mathop\mathrm{span}(\bm{x}_{1},\ldots,\bm{x}_{m})
=
\left\{\left.\bigoplus_{i=1}^{m}c_{i}\bm{x}_{i}\right|c_{1},\ldots,c_{m}\in\mathbb{X}\right\}.
$$

An example of the linear span of two vectors $\bm{x}_{1}$ and $\bm{x}_{2}$ in $\mathbb{R}_{\max,+}^{2}$ is given in Fig.~\ref{F-LS}, where $\mathop\mathrm{span}(\bm{x}_{1},\bm{x}_{2})$ is a region between two thick lines going through the end points of the vectors at $45^{\circ}$ angle to the axes of Cartesian coordinates.
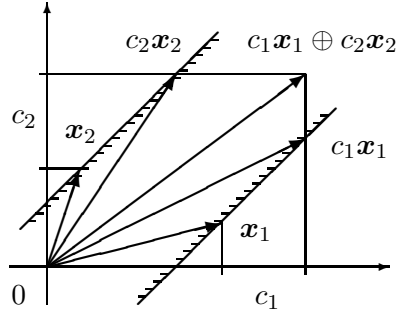
\begin{figure}[ht]
\setlength{\unitlength}{1mm}
\begin{center}
\begin{picture}(50,40)

\put(0,5){\vector(1,0){50}}
\put(5,0){\vector(0,1){40}}

\put(5,5){\thicklines\vector(1,3){4.25}}
\put(5,5){\thicklines\vector(2,3){17}}

\put(5,5){\thicklines\vector(4,1){23}}
\put(5,5){\thicklines\vector(2,1){34}}

\put(1.5,10){\thicklines\line(1,1){26}}
\multiput(2.5,11)(1,1){25}{\line(1,0){1}}

\put(17,0){\thicklines\line(1,1){26}}
\multiput(18,1)(1,1){25}{\line(-1,0){1}}

\put(39,30.5){\line(-1,0){35}}
\put(39,30.5){\line(0,-1){26.5}}

\put(10,18){\line(-1,0){6}}
\put(28,11){\line(0,-1){7}}

\put(5,5){\thicklines\vector(4,3){34}}

\put(0,0){$0$}

\put(7,22){$\bm{x}_{2}$}
\put(0,24){$c_{2}$}
\put(15,34){$c_{2}\bm{x}_{2}$}

\put(30,9){$\bm{x}_{1}$}
\put(32,0){$c_{1}$}
\put(42,20){$c_{1}\bm{x}_{1}$}

\put(31,34){$c_{1}\bm{x}_{1}\oplus c_{2}\bm{x}_{2}$}

\end{picture}
\end{center}
%\vspace{-2ex}
\caption{A linear span of two vectors in $\mathbb{R}_{\max,+}^{2}$.}\label{F-LS}
%\vspace{-1ex}
\end{figure}

For any nonzero column vector $\bm{x}=(x_{i})\in\mathbb{X}^{n}$, its pseudo-inverse is a row vector $\bm{x}^{-}=(x_{i}^{-})$, where $x_{i}^{-}=x_{i}^{-1}$ if $x_{i}\ne\mathbb{0}$, and $x_{i}^{-}=\mathbb{0}$ otherwise.

If $\bm{x}$ is any nonzero vector, then $\bm{x}^{-}\bm{x}=\mathbb{1}$. 

For all regular vectors $\bm{x},\bm{y}\in\mathbb{X}_{+}^{n}$, the componentwise inequality $\bm{x}\leq\bm{y}$ implies $\bm{x}^{-}\geq\bm{y}^{-}$.

The distance between any two regular vectors $\bm{x}=(x_{i})$ and $\bm{y}=(y_{i})$ is given by
$$
\rho(\bm{x},\bm{y})
=
\bm{y}^{-}\bm{x}\oplus\bm{x}^{-}\bm{y}.
$$

In the semimodule $\mathbb{R}_{\max,+}^{n}$, this distance becomes
$$
\rho(\bm{x},\bm{y})
=
\max_{1\leq i\leq n}|x_{i}-y_{i}|,
$$
and so coincides with the classical Chebyshev metric.

\subsection{Algebra of Matrices}

For conforming matrices $A=(a_{ij})$, $B=(b_{ij})$, and $C=(c_{ij})$, matrix addition and multiplication together with  multiplication by a scalar $c\in\mathbb{X}$ are performed according to the formulas
\begin{gather*}
\{A\oplus B\}_{ij}
=
a_{ij}\oplus b_{ij},
\qquad
\{B C\}_{ij}
=
\bigoplus_{k}b_{ik}c_{kj},
\\
\{cA\}_{ij}=ca_{ij}.
\end{gather*}

A matrix with all zero entries is a zero matrix which is denoted by $\mathbb{0}$.

A matrix is called regular if it has no zero rows.

Consider the set of square matrices $\mathbb{X}^{n\times n}$. Any matrix that has all off-diagonal entries equal to $\mathbb{0}$ is called diagonal. If a matrix has all entries above (below) the diagonal equal to $\mathbb{0}$, then it is triangular.

The diagonal matrix with all diagonal entries equal to $\mathbb{1}$ is the identity matrix denoted by $I$.

For any matrix $A\ne\mathbb{0}$ and an integer $p>0$, the power notation is routinely defined as
$$
A^{0}=I,
\qquad
A^{p}=A^{p-1}A=AA^{p-1}.
$$

With respect to matrix addition and multiplication, $\mathbb{X}^{n\times n}$ is an idempotent semiring with identity.

A matrix is reducible if it can be put in a block-triangular form by simultaneous permutations of rows and columns. Otherwise the matrix is irreducible.

For every matrix $A=(a_{ij})$, its trace is given by
$$
\mathop\mathrm{tr}A
=
\bigoplus_{i=1}^{n}a_{ii}.
$$

%For any regular vectors $\bm{x},\bm{y}$ it holds that
%$$
%\bm{x}\bm{y}^{-}
%\geq
%(\bm{x}^{-}\bm{y})^{-1}I.
%$$

\subsection{Linear Operators and Linear Equations}

Every matrix $A\in\mathbb{X}^{m\times n}$ defines a mapping from the semimodule $\mathbb{X}^{n}$ to the semimodule $\mathbb{X}^{m}$. Since for any vectors $\bm{x},\bm{y}\in\mathbb{X}^{n}$ and scalar $c\in\mathbb{X}$, it holds that
$$
A(\bm{x}\oplus\bm{y})
=
A\bm{x}\oplus A\bm{y},
\qquad
A(c\bm{x})
=
cA\bm{x},
$$
the mapping can be considered as a linear operator.

For given matrices $A,C\in\mathbb{X}^{m\times n}$ and vectors $\bm{b},\bm{d}\in\mathbb{X}^{m}$, a general linear equation in the unknown vector $\bm{x}\in\mathbb{X}^{n}$ is written in the form
$$
A\bm{x}\oplus\bm{b}
=
C\bm{x}\oplus\bm{d}.
$$ 

Since there is no additive inverse, one cannot rearrange the equation in such a way that all terms involving the unknown $\bm{x}$ are brought to one side of the equation while those without $\bm{x}$ go to another side.

Many practical problems reduce to solution of the following particular cases of the general equation
$$
A\bm{x}
=
\bm{d},
\qquad
A\bm{x}
=
\bm{x}.
$$

By analogy with linear integral equations, the above two equations are respectively referred to as that of the first kind and that of the second kind. The last equation is also known in the literature as the homogeneous Bellman equation.

Along with the equations, one can consider first- and second-kind inequalities that have the form
$$
A\bm{x}
\leq
\bm{d},
\qquad
A\bm{x}
\leq
\bm{x}.
$$

\section{Preliminary Results}

In this section, we present results from \cite{Krivulin05Onsolution,Krivulin06Solution,Krivulin09Methods} to be used below in the idempotent algebra based analysis of location problems. These results seem to have a less complicated form and provide better geometrical interpretation than similar algebraic solutions offered in the literature including  \cite{Baccelli92Synchronization,Cuninghamegreen94Minimax,Kolokoltsov97Idempotent,Litvinov98Correspondence,Golan03Semirings,Heidergott05Maxplus,Butkovic10Maxlinear,Gaubert08Cyclic,Butkovic09Onsome}.

\subsection{The Equation of the First Kind}

Given a matrix $A\in\mathbb{X}^{m\times n}$ and a vector $\bm{d}\in\mathbb{X}^{m}$, the problem is to find all solutions $\bm{x}\in\mathbb{X}^{n}$ of the equation
\begin{equation}
A\bm{x}
=
\bm{d}.
\label{E-Axd}
\end{equation}

A solution $\bm{x}_{1}$ to equation \eqref{E-Axd} is called the maximum solution if $\bm{x}_{1}\geq\bm{x}$ for all solutions $\bm{x}$ of \eqref{E-Axd}.

We present a solution to equation \eqref{E-Axd} based on the analysis of the distance between vectors in $\mathbb{X}^{m}$.

Let $\bm{a}_{i}$ represent column $i$ in the matrix $A$ for each $i=1,\ldots,m$. Consider a problem of evaluating the minimum distance in the sense of the metric $\rho$ from the vector $\bm{d}$ to the linear span of these columns.

Since we can write
$$
\mathop\mathrm{span}(\bm{a}_{1},\ldots,\bm{a}_{m})
\\
=
\{A\bm{x}|\bm{x}\in\mathbb{X}^{n}\},
$$
the problem is to find vectors $\bm{x}\in\mathbb{X}^{n}$ that minimize
$$
\rho(A\bm{x},\bm{d})
=
(A\bm{x})^{-}\bm{d}\oplus\bm{d}^{-}A\bm{x}.
$$

The next result gives the solution to the problem when both the matrix $A$ and the vector $\bm{d}$ are regular. 

\begin{lemma}\label{L-minAxd}
Suppose $A\in\mathbb{X}^{m\times n}$ is a regular matrix, $\bm{d}\in\mathbb{X}^{m}$ is a regular vector, and define
$$
\Delta
=
\sqrt{(A(\bm{d}^{-}A)^{-})^{-}\bm{d}}.
$$

Then it holds that
$$
\min_{\bm{x}\in\mathbb{X}_{+}^{n}}\rho(A\bm{x},\bm{d})=\Delta
$$
with the minimum attained at $\bm{x}_{0}=\Delta(\bm{d}^{-}A)^{-}$.
\end{lemma}

Fig.~\ref{F-Lb} presents examples of the mutual arrangement of the linear span of the columns in a matrix $A$ and a vector $\bm{d}$ in $\mathbb{R}_{\max,+}^{2}$. In the case when $\Delta>\mathbb{1}$, the minimum distance to the nearest vector of the linear span is attained at the vector $\bm{y}=\Delta A(\bm{d}^{-}A)^{-}$. 
\begin{figure}[ht]
\setlength{\unitlength}{1mm}
\begin{center}
\begin{picture}(35,45)

\put(0,5){\vector(1,0){35}}
\put(5,0){\vector(0,1){45}}

\put(5,5){\thicklines\vector(1,4){4}}

\put(5,5){\thicklines\vector(2,-1){5}}

\put(0,12){\line(1,1){25}}
\put(0,12){\thicklines\line(1,1){25}}
\multiput(1,13)(1,1){24}{\line(1,0){1}}

\put(7.5,0){\line(1,1){25}}
\put(7.5,0){\thicklines\line(1,1){25}}
\multiput(8,0.5)(1,1){25}{\line(-1,0){1}}

\put(5,5){\thicklines\vector(3,4){18}}

\put(13,0){$\bm{a}_{1}$}
\put(7,26){$\bm{a}_{2}$}
\put(24,30){$\bm{d}$}

\put(9,38){$\Delta=\mathbb{1}$}

\end{picture}
\hspace{20\unitlength}
\begin{picture}(40,45)

\put(0,5){\vector(1,0){40}}
\put(5,0){\vector(0,1){45}}

\put(5,5){\thicklines\vector(1,4){4}}

\put(5,5){\thicklines\vector(2,-1){5}}

\put(0,12){\line(1,1){25}}
\put(0,12){\thicklines\line(1,1){25}}
\multiput(1,13)(1,1){24}{\line(1,0){1}}

\put(7.5,0){\line(1,1){25}}
\multiput(8,0.5)(1,1){25}{\line(-1,0){1}}
\put(7.5,0){\thicklines\line(1,1){25}}

\put(27.5,20){\line(1,-1){7.5}}

\put(5,5){\thicklines\vector(3,2){22.5}}

%\put(5,5){\thicklines\vector(4,1){30}}
\put(5,5){\thicklines\vector(4,1){30}}

\multiput(27.5,20)(0,-2.8){6}{\line(0,-1){2}}

\multiput(35,12.5)(0,-3.2){3}{\line(0,-1){2.2}}

\put(13,0){$\bm{a}_{1}$}
\put(7,26){$\bm{a}_{2}$}
\put(35,14){$\bm{d}$}

\put(25,23){$\bm{y}$}

\put(9,38){$\Delta>\mathbb{1}$}

\put(29.5,0){$\Delta$}

\end{picture}
\end{center}
%\vspace{-2ex}
\caption{The set $\mathop\mathrm{span}(\bm{a}_{1},\bm{a}_{2})$ and the vector $\bm{d}$ in $\mathbb{R}_{\max,+}^{2}$ when $\Delta=\mathbb{1}$ (left) and $\Delta>\mathbb{1}$ (right).}\label{F-Lb}
%\vspace{-1ex}
\end{figure}
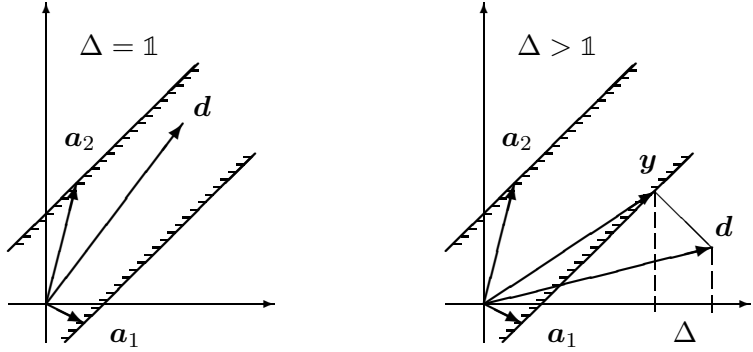

As a consequence, we get the following result.
\begin{theorem}\label{T-EAxb}
A solution of equation \eqref{E-Axd} exists if and only if $\Delta=\mathbb{1}$. If solvable, the equation has the maximum solution given by
$$
\bm{x}
=
(\bm{d}^{-}A)^{-}.
$$
\end{theorem}

The general solution to equation \eqref{E-Axd} with arbitrary matrix $A$ and vector $d$ is considered in \cite{Krivulin05Onsolution,Krivulin09Methods}.

\subsection{Second-Kind Equations and Inequalities}

Suppose $A\in\mathbb{X}^{n\times n}$ is a given matrix, and $\bm{x}\in\mathbb{X}^{n}$ is an unknown vector. We examine the equation
\begin{equation}
A\bm{x}
=
\bm{x},
\label{E-Axbx}
\end{equation}
and the inequality
\begin{equation}
A\bm{x}
\leq
\bm{x},
\label{I-Axbx}
\end{equation}

To solve both equation \eqref{E-Axbx} and inequality \eqref{I-Axbx} we propose an approach based on the use of a function $\mathop\mathrm{Tr}(A)$ that takes each square matrix $A$ to a scalar according to the definition
$$
\mathop\mathrm{Tr}(A)
=
\bigoplus_{m=1}^{n}\mathop\mathrm{tr} A^{m}.
$$

The function is exploited to examine whether the equation has a unique solution, many solutions, or no solution, and so may play the role of the determinant in conventional linear algebra.

The solution involves evaluating matrices $A^{\ast}$, $A^{\times}$, and $A^{+}$. The first two are given by 
%\begin{align*}
%A^{\ast}
%&=
%I\oplus A\oplus\cdots\oplus A^{n-1},
%\\
%A^{\times}
%&=
%AA^{\ast}
%=
%A\oplus\cdots\oplus A^{n}.
%\end{align*}
$$
A^{\ast}
=
I\oplus A\oplus\cdots\oplus A^{n-1},
\qquad
A^{\times}
=
AA^{\ast}.
$$

The matrix $A^{+}$ is constructed as follows. Let $\bm{a}_{i}^{\times}$ be column $i$ in $A^{\times}$, and $a_{ii}^{\times}$ be its diagonal element. First we replace each column $\bm{a}_{i}^{\times}$ with that defined as 
$$
\bm{a}_{i}^{+}
=
\begin{cases}
\bm{a}_{i}^{\times}, & \text{if $a_{ii}^{\times}=\mathbb{1}$},
\\
\mathbb{0}, & \text{otherwise}.
\end{cases}
$$

Furthermore, the set of columns $\bm{a}_{i}^{+}$ is reduced by removing those columns, if any, that are linearly dependent on others. Finally, the rest columns are put together to form the matrix $A^{+}$.

The general solutions to both equation and inequality of the second kind in the case of irreducible matrices are given by the following results.
\begin{theorem}\label{T-IMNHEGS}
Let $A$ be an irreducible matrix, and $\bm{x}$ be the solution of equation \eqref{E-Axbx} with the matrix $A$.

Then the following statements hold:
\begin{enumerate}
\item[1)] if $\mathop\mathrm{Tr}(A)=\mathbb{1}$, then $\bm{x}=A^{+}\bm{v}$ for any vector $\bm{v}$ of appropriate size;
\item[2)] if $\mathop\mathrm{Tr}(A)\ne\mathbb{1}$, then there is only the trivial solution $\bm{x}=\bm{\mathbb{0}}$.
\end{enumerate}
\end{theorem}

Fig.~\ref{F-GSEAxx} gives examples of solutions to equations \eqref{E-Axbx} in $\mathbb{R}_{\max,+}^{2}$ for some particular matrices $A=(\bm{a}_{1},\bm{a}_{2})$. In the left example, the solution set is depicted by a thick line drawn through the end point of the vector $\bm{a}_{2}$. The solution on the right takes the form of a strip between thick lines going through $\bm{a}_{1}$ and $\bm{a}_{2}$. 
\begin{figure}[ht]
\setlength{\unitlength}{1mm}
\begin{center}
\begin{picture}(33,50)

\put(0,23){\vector(1,0){33}}
\put(14,0){\vector(0,1){50}}

\put(14,23){\thicklines\vector(-1,0){10}}
\put(1,20){\thicklines\line(1,1){25}}

\put(14,23){\thicklines\vector(-1,-4){5}}
\put(7,1){\line(1,1){25}}

\put(14,23){\thicklines\vector(1,2){10}}

\multiput(24,43)(0,-3.1){7}{\line(0,-1){2}}

\multiput(24,43)(-3,0){4}{\line(-1,0){2}}

\put(0,26){$\bm{a}_{2}$}
\put(3,4){$\bm{a}_{1}$}

\put(27,41){$\bm{x}$}

\put(20,19){$x_{1}$}

\put(8,43){$x_{2}$}

\put(10,19){$0$}

\end{picture}
\hspace{20\unitlength}
\begin{picture}(33,50)

\put(0,23){\vector(1,0){33}}
\put(16,0){\vector(0,1){50}}

\put(16,23){\thicklines\vector(-1,0){12}}
\put(1,20){\thicklines\line(1,1){25}}
\multiput(2,21)(1,1){24}{\line(1,0){1}}

\put(16,23){\thicklines\vector(0,-1){8}}
\put(7,6){\thicklines\line(1,1){24}}
\multiput(8,7)(1,1){23}{\line(-1,0){1}}

\put(16,23){\thicklines\vector(2,3){12}}

\multiput(28,41)(0,-2.8){7}{\line(0,-1){2}}

\multiput(28,41)(-2.8,0){5}{\line(-1,0){2}}

\put(0,26){$\bm{a}_{2}$}
\put(18,12){$\bm{a}_{1}$}

\put(29,42){$\bm{x}$}

\put(27,19){$x_{1}$}

\put(10,41){$x_{2}$}

\put(12,19){$0$}

\end{picture}
\end{center}
%\vspace{-2ex}
\caption{Solution sets for second-kind equations in $\mathbb{R}_{\max,+}^{2}$.}\label{F-GSEAxx}
%\vspace{-1ex}
\end{figure}
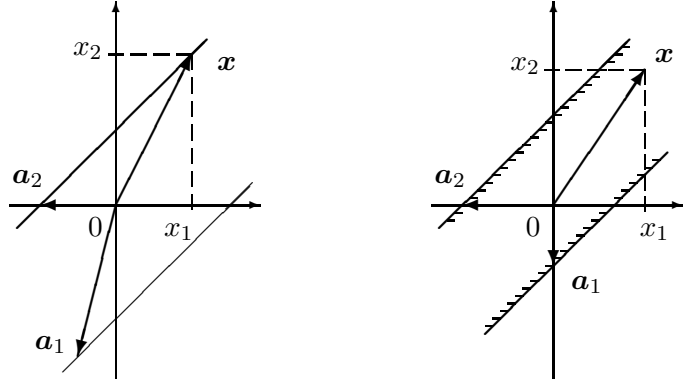

\begin{lemma}
Let $A$ be an irreducible matrix, and $\bm{x}$ be the solution of inequality \eqref{I-Axbx} with the matrix $A$.

Then the following statements hold:
\begin{enumerate}
\item[1)] if $\mathop\mathrm{Tr}(A)\leq\mathbb{1} $, then $\bm{x}=A^{\ast}\bm{v}$ for any vector $\bm{v}$ of appropriate size;
\item[2)] if $\mathop\mathrm{Tr}(A)>\mathbb{1}$, then there is only the trivial solution $\bm{x}=\bm{\mathbb{0}}$.
\end{enumerate}
\end{lemma}

Fig.~\ref{F-GSIAxx} shows the solution of the inequality with the same matrix as in the left example in Fig.~\ref{F-GSEAxx}. The solution set forms a region that is bounded by the slanted thick lines. Note that in the case of the matrix in the right example of Fig.~\ref{F-GSEAxx}, the solution sets of both equation and inequality coincide. 
\begin{figure}[ht]
\setlength{\unitlength}{1mm}
\begin{center}
\begin{picture}(40,50)

\put(0,23){\vector(1,0){40}}
\put(14,0){\vector(0,1){50}}

\put(14,23){\thicklines\vector(-1,0){10}}
\put(1,20){\thicklines\line(1,1){25}}
\multiput(2,21)(1,1){24}{\line(1,0){1}}

\put(14,23){\thicklines\vector(-1,-4){5}}
\put(7,1){\line(1,1){25}}

\put(10,0){\thicklines\line(1,1){25}}
\multiput(10,0)(1,1){26}{\line(-1,0){1}}

\put(14,23){\thicklines\vector(2,3){10}}

\multiput(24,38)(0,-2.8){6}{\line(0,-1){2}}

\multiput(24,38)(-3,0){4}{\line(-1,0){2}}

\put(0,26){$\bm{a}_{2}$}
\put(3,4){$\bm{a}_{1}$}

\put(25,39){$\bm{x}$}

\put(20,19){$x_{1}$}

\put(8,38){$x_{2}$}

\put(10,19){$0$}

\end{picture}
\end{center}
%\vspace{-2ex}
\caption{Solution set for a second-kind inequality in $\mathbb{R}_{\max,+}^{2}$.}\label{F-GSIAxx}
%\vspace{-1ex}
\end{figure}
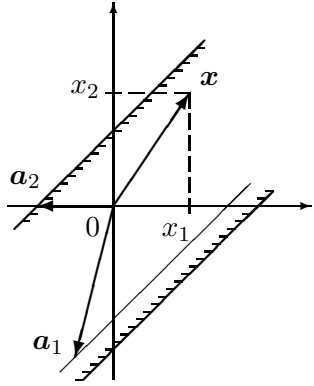

Related results for the case of arbitrary matrices can be found in \cite{Krivulin06Solution,Krivulin09Methods}.

\section{Optimization Problems}

Given a matrix $A\in\mathbb{X}^{m\times n}$ and vectors $\bm{b},\bm{c}\in\mathbb{X}^{m}$, consider a problem that is to find
\begin{equation}
\min_{\bm{x}\in\mathbb{X}_{+}^{n}}((A\bm{x})^{-}\bm{b}\oplus\bm{c}^{-}A\bm{x}).
\label{P-minphix}
\end{equation}

Note that a particular case of the problem when $\bm{c}=\bm{b}$ arises in the previous section in the context of the solution of first-kind equations.

The next result offers a solution to problem \eqref{P-minphix}.
\begin{theorem}
\label{T-AxbcAx}
Suppose that $A$ is a regular matrix, $\bm{b}$ and $\bm{c}$ are regular vectors, and assume that
$$
\Delta
=
\sqrt{(A(\bm{c}^{-}A)^{-})^{-}\bm{b}}.
$$

Then it holds that
\begin{equation}
\min_{\bm{x}\in\mathbb{X}_{+}^{n}}((A\bm{x})^{-}\bm{b}\oplus\bm{c}^{-}A\bm{x})
=
\Delta,
\label{E-AxbcAx}
\end{equation}
where the minimum is attained at the vector
$$
\bm{x}
=
\Delta(\bm{c}^{-}A)^{-}.
$$
\end{theorem}
\begin{proof}
We first verify that $\Delta$ is a lower bound for the objective function in \eqref{P-minphix}, and then show that the function reaches the bound when $\bm{x}=\Delta(\bm{c}^{-}A)^{-}$.
  
Take any vector $\bm{x}\in\mathbb{X}_{+}^{n}$ and consider
$$
r
=
(A\bm{x})^{-}\bm{b}\oplus\bm{c}^{-}A\bm{x}.
$$

From the last equality we have two inequalities
$$
r
\geq
\bm{c}^{-}A\bm{x},
\qquad
r
\geq
(A\bm{x})^{-}\bm{b}.
$$

Right multiplication of the first inequality by $\bm{x}^{-}$ together with the obvious inequality $\bm{x}\bm{x}^{-}\geq I$ give
$$
r\bm{x}^{-}
\geq
\bm{c}^{-}A\bm{x}\bm{x}^{-}
\geq
\bm{c}^{-}A.
$$

Furthermore, by pseudo-inverting both sides, we get the inequality $\bm{x}\leq r(\bm{c}^{-}A)^{-}$. Left multiplication by $A$ followed by another application of pseudo-inversion leads to
$$
(A\bm{x})^{-}
\geq
r^{-1}(A(\bm{c}^{-}A)^{-})^{-}.
$$

Substitution into the second inequality results in
$$
r
\geq
r^{-1}(A(\bm{c}^{-}A)^{-})^{-}\bm{b}
=
r^{-1}\Delta^{2}
$$
and consequently, in the inequality $r\geq\Delta$.

It remains to verify that we have $r=\Delta$ when putting $\bm{x}=\Delta(\bm{c}^{-}A)^{-}$. Indeed, in this case we have
$$
r
=
(A\bm{x})^{-}\bm{b}\oplus\bm{c}^{-}A\bm{x}
=
\Delta^{-1}(A(\bm{c}^{-}A)^{-})^{-}\bm{b}\oplus\Delta\bm{c}^{-}A(\bm{c}^{-}A)^{-}
=
\Delta
\oplus
\Delta
=
\Delta,
$$
that concludes the proof.
\end{proof}

Now we present a useful property of the solution and consider particular cases when an extended solution set appears to exist. 
\begin{corollary}
\label{C-AxbcAx}
Under the assumptions of Theorem~\ref{T-AxbcAx} the following statements hold:
\begin{enumerate}
\item[(i)] any vector $\bm{x}$ that gives the minimum in problem \eqref{P-minphix} satisfies the inequality
\begin{equation}
\label{I-DbAxDc}
\Delta^{-1}\bm{b}
\leq
A\bm{x}
\leq
\Delta\bm{c};
\end{equation}
\item[(ii)] 
if $\bm{u}$ is a solution of the equation $A\bm{u}=\bm{b}$, then the minimum in \eqref{P-minphix} is attained at $\bm{x}=\Delta^{-1}\bm{u}$;
\item[(iii)]
if $\bm{v}$ is a solution of the equation $A\bm{v}=\bm{c}$, then the minimum in \eqref{P-minphix} is attained at $\bm{x}=\Delta\bm{v}$.
\end{enumerate}
\end{corollary}
\begin{proof}
To verify statement (i) suppose $\bm{x}$ is a vector that solves problem \eqref{P-minphix}. From the equality
$$
(A\bm{x})^{-}\bm{b}\oplus\bm{c}^{-}A\bm{x}
=
\Delta
$$ 
we get two inequalities
$$
(A\bm{x})^{-}\bm{b}
\leq
\Delta,
\qquad
\bm{c}^{-}A\bm{x}
\leq
\Delta,
$$
and note that both vectors $\bm{b}$ and $\bm{c}$ are regular.

With the same technique as above, from the first inequality we have
$$
(A\bm{x})^{-}\leq(A\bm{x})^{-}\bm{b}\bm{b}^{-}\leq\Delta\bm{b}^{-},
$$
and then arrive at the inequality  $A\bm{x}\geq\Delta^{-1}\bm{b}$, which is the left part of \eqref{I-DbAxDc}. To get the right part, we take the second inequality and write
$$
A\bm{x}\leq\bm{c}\bm{c}^{-}A\bm{x}\leq\Delta\bm{c}.
$$

To prove statements (ii) and (iii), we first show that $\Delta^{2}\geq\bm{c}^{-}\bm{b}$. We have
$$
A(\bm{c}^{-}A)^{-}\leq\bm{c}\bm{c}^{-}A(\bm{c}^{-}A)^{-}=\bm{c}
$$
and thus
$$
\Delta^{2}
=
(A(\bm{c}^{-}A)^{-})^{-}\bm{b}
\geq
\bm{c}^{-}\bm{b}.
$$

Now suppose $A\bm{u}=\bm{b}$ and take $\bm{x}=\Delta^{-1}\bm{u}$. With the above inequality, we get
$$
(A\bm{x})^{-}\bm{b}\oplus\bm{c}^{-}A\bm{x}
=
\Delta\bm{b}^{-}\bm{b}\oplus\Delta^{-1}\bm{c}^{-}\bm{b}
=
\Delta.
$$

In the same way we assume that $A\bm{v}=\bm{c}$ and put $\bm{x}=\Delta\bm{v}$. Substitution of $\bm{x}$ gives
$$
(A\bm{x})^{-}\bm{b}\oplus\bm{c}^{-}A\bm{x}
=
\Delta^{-1}\bm{c}^{-}\bm{b}\oplus\Delta\bm{c}^{-}\bm{c}
=
\Delta.
\qedhere
$$
\end{proof}

Now we give the solution to problem~\eqref{P-minphix} in a particular case when $A=I$.

\begin{theorem}
\label{T-xbcx}
Suppose that $\bm{b}$ and $\bm{c}$ are regular vectors, and assume that
$$
\Delta=(\bm{c}^{-}\bm{b})^{1/2}.
$$

Then it holds that
$$
\min_{\bm{x}\in\mathbb{X}_{+}^{n}}(\bm{x}^{-}\bm{b}\oplus\bm{c}^{-}\bm{x})
=
\Delta
$$
with the minimum attained at any vector $\bm{x}$ such that
$$
\Delta^{-1}\bm{b}
\leq
\bm{x}
\leq
\Delta\bm{c}.
$$
\end{theorem}

\begin{proof}
The statements of the Theorem directly follow from \eqref{E-AxbcAx} and \eqref{I-DbAxDc} provided that $A=I$.
\end{proof}

\section{Unconstrained Location Problem}

In this section we examine a minimax single facility location problem with Chebyshev distance when no constraints are imposed on the feasible location area.

Given $m$ vectors $\bm{r}_{i}=(r_{1i},\ldots,r_{ni})^{T}\in\mathbb{R}^{n}$ and constants $w_{i}\in\mathbb{R}$, $i=1,\ldots,m$, the problem is to determine the vectors $\bm{x}\in\mathbb{R}^{n}$ that provide
\begin{equation}
\min_{\bm{x}\in\mathbb{R}^{n}}\max_{1\leq i\leq m}(\rho(\bm{r}_{i},\bm{x})+w_{i}).
\label{P-Chebyshev}
\end{equation}

The problem is known as the unweighted Rawls problem with addends \cite{Hansen81Constrained}. In accordance with the nomenclature of \cite{Elzinga72Geometrical}, it can be referred to as the multidimensional Chebyshev Messenger Boy Problem.

It is not difficult to solve the problem on the plane by using geometric arguments \cite{Sule01Logistics,Moradi09Single}. Below we give a new algebraic solution that is based on representation of the problem in terms of the semifield $\mathbb{R}_{\max,+}$, and application of the results from the previous section.

First we denote the objective function in problem \eqref{P-Chebyshev} by $\varphi(\bm{x})$ and write
$$
\varphi(\bm{x})
=
\bigoplus_{i=1}^{m}w_{i}\rho(\bm{r}_{i},\bm{x}).
$$

Furthermore, we introduce the vectors
$$
\bm{p}
=
w_{1}\bm{r}_{1}\oplus\cdots\oplus w_{m}\bm{r}_{m},
\qquad
\bm{q}^{-}
=
w_{1}\bm{r}_{1}^{-}\oplus\cdots\oplus w_{m}\bm{r}_{m}^{-}.
$$

Writing the metric in terms of $\mathbb{R}_{\max,+}$, we have
$$
\varphi(\bm{x})
=
\bigoplus_{i=1}^{m}w_{i}(\bm{x}^{-}\bm{r}_{i}\oplus\bm{r}_{i}^{-}\bm{x})
=
\bm{x}^{-}\bm{p}
\oplus
\bm{q}^{-}\bm{x},
$$
and then represent problem \eqref{P-Chebyshev} as
\begin{equation}\label{P-Chebyshev1}
\min_{\bm{x}\in\mathbb{R}^{n}}(\bm{x}^{-}\bm{p}\oplus\bm{q}^{-}\bm{x}).
\end{equation}

Application of Theorem~\ref{T-xbcx} leads us to the following results which conform with that in \cite{Krivulin11Algebraic,Krivulin11Analgebraic}.
\begin{theorem}
\label{T-Chebyshev}
The minimum in problem \eqref{P-Chebyshev1} is given by
$$
\Delta
=
(\bm{q}^{-}\bm{p})^{1/2},
$$
with the minimum attained at any vector $\bm{x}$ such that
$$
\Delta^{-1}\bm{p}
\leq
\bm{x}
\leq
\Delta\bm{q}.
$$
\end{theorem}

With the usual notation, we can reformulate the statement of Theorem~\ref{T-Chebyshev} as follows.
\begin{corollary}
Suppose that for each $i=1,\ldots,n$
$$
p_{i}
=
\max(r_{i1}+w_{1},\ldots,r_{im}+w_{m}),
\qquad
q_{i}
=
\min(r_{i1}-w_{1},\ldots,r_{im}-w_{m}).
$$

Then the minimum in \eqref{P-Chebyshev} is given by
$$
\Delta
=
\max(p_{1}-q_{1},\ldots,p_{n}-q_{n})/2,
$$
and attained at any vector $\bm{x}=(x_{i})$ with elements
$$
p_{i}-\Delta
\leq
x_{i}
\leq
q_{i}+\Delta,
\quad
i=1,\ldots,n.
$$
\end{corollary}

An illustration of the solution in $\mathbb{R}^{2}$ for two problems with all $w_{i}=0$ is provided in Fig.~\ref{F-Chebyshev1}, where the given points are represented by thick dots. The solution involves drawing a minimal upright rectangle enclosing all points. The solution set is represented by thick line segments that go through the centers of the long sides in the rectangle between two slanting lines drawn through the vertices of the rectangle.

\begin{figure}[ht]
\setlength{\unitlength}{1mm}
\begin{center}
\begin{picture}(35,60)

\put(0,5){\vector(1,0){32}}
\put(5,0){\vector(0,1){55}}

\put(8,48){\line(0,-1){44}}
\put(23,48){\line(0,-1){44}}

\put(23,10){\line(-1,0){19}}
\put(23,48){\line(-1,0){19}}

\put(8,10){\vector(1,1){19}}
%\multiput(15,10)(5.4,5.4){4}{\line(1,1){3.6}}

\put(23,48){\vector(-1,-1){19}}
%\multiput(30,50)(-5.4,-5.4){4}{\line(-1,-1){3.6}}

\put(4,29){\linethickness{2pt}\line(1,0){23}}

\put(8,43){\circle*{1.5}}
\put(11,10){\circle*{1.5}}
%\put(12,24){\circle*{1.5}}
\put(16,33){\circle*{1.5}}
\put(17,48){\circle*{1.5}}
%\put(18,50){\circle*{1.5}}
\put(20,41){\circle*{1.5}}
\put(21,13){\circle*{1.5}}
\put(23,22){\circle*{1.5}}
%\put(19,28){\circle*{1.5}}
\put(15,23){\circle*{1.5}}

\put(7,1){$q_{1}$}
\put(22,1){$p_{1}$}

%\put(30,1){$x_{1}$}

\put(-0.5,48){$p_{2}$}
\put(-0.5,10){$q_{2}$}

%\put(-0.5,55){$x_{2}$}

\end{picture}
\hspace{20\unitlength}
\begin{picture}(40,60)

\put(0,5){\vector(1,0){40}}
\put(5,0){\vector(0,1){55}}

\put(8,35){\line(0,-1){31}}
\put(34,35){\line(0,-1){31}}

\put(34,15){\line(-1,0){30}}
\put(34,35){\line(-1,0){30}}

\put(8,15){\vector(1,1){13}}
%\multiput(10,15)(5.5,5.5){3}{\line(1,1){3.6}}

\put(34,35){\vector(-1,-1){13}}
%\multiput(40,35)(-5.5,-5.5){3}{\line(-1,-1){3.6}}

\put(21,22){\linethickness{2pt}\line(0,1){6}}

\put(13,33){\circle*{1.5}}
\put(16,15){\circle*{1.5}}
\put(8,24){\circle*{1.5}}
%\put(21,23){\circle*{1.5}}
\put(22,35){\circle*{1.5}}
\put(29,32){\circle*{1.5}}
\put(18,30){\circle*{1.5}}
%\put(38,17){\circle*{1.5}}
\put(28,22){\circle*{1.5}}
%\put(35,30){\circle*{1.5}}
\put(34,19){\circle*{1.5}}

\put(7,1){$q_{1}$}
\put(32,1){$p_{1}$}

%\put(38,1){$x_{1}$}

\put(-0.5,35){$p_{2}$}
\put(-0.5,15){$q_{2}$}

%\put(-0.5,55){$x_{2}$}

\end{picture}
\end{center}
%\vspace{-2ex}
\caption{Solutions in $\mathbb{R}^{2}$ when all $w_{i}=0$.}\label{F-Chebyshev1}
%\vspace{-1ex}
\end{figure}
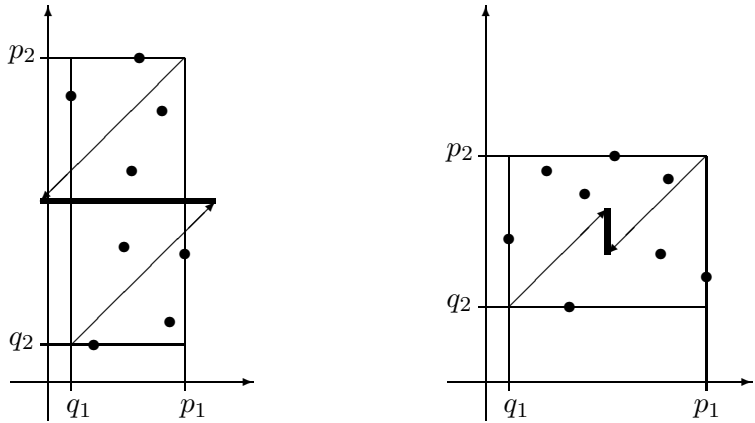

Fig.~\ref{F-Chebyshev3} offers an illustration to a problem with arbitrary constants $w_{i}$. Together with the solution of this problem (right picture), we also give the solution to a corresponding problem that has all $w_{i}$ set to zero (left picture). To get the solution, we first replace each given point with two new points shown with empty circles, and then draw their related minimal rectangle.

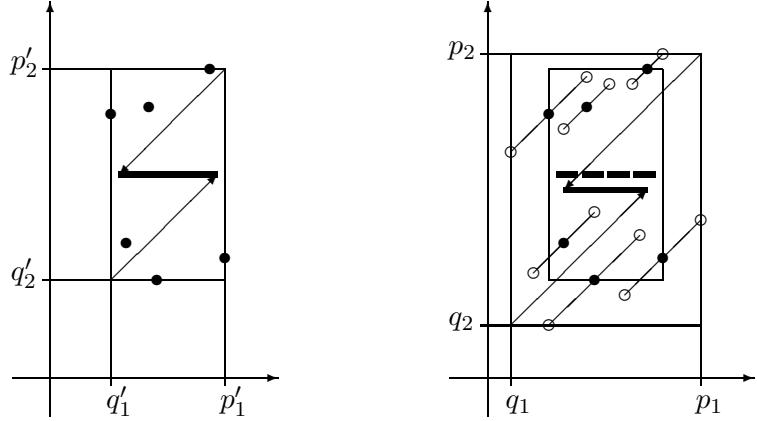
\begin{figure}[ht]
\setlength{\unitlength}{1mm}
\begin{center}
\begin{picture}(35,60)

\put(0,5){\vector(1,0){35}}
\put(5,0){\vector(0,1){55}}

\put(13,46){\line(0,-1){42}}
\put(28,46){\line(0,-1){42}}

\put(28,18){\line(-1,0){24}}
\put(28,46){\line(-1,0){24}}

\put(13,18){\vector(1,1){14}}
%\multiput(15,18)(5.2,5.2){3}{\line(1,1){3.6}}

\put(28,46){\vector(-1,-1){14}}
%\multiput(30,46)(-5.2,-5.2){3}{\line(-1,-1){3.6}}

\put(14,32){\linethickness{2pt}\line(1,0){13}}

\put(15,23){\circle*{1.5}}

\put(18,41){\circle*{1.5}}

\put(26,46){\circle*{1.5}}

\put(13,40){\circle*{1.5}}

\put(28,21){\circle*{1.5}}

\put(19,18){\circle*{1.5}}

\put(12,1){$q_{1}^{\prime}$}
\put(27,1){$p_{1}^{\prime}$}

%\put(45,1){$x_{1}$}

\put(-0.5,46){$p_{2}^{\prime}$}
\put(-0.5,18){$q_{2}^{\prime}$}

%\put(0,55){$x_{2}$}

\end{picture}
\hspace{20\unitlength}
\begin{picture}(40,60)

\put(0,5){\vector(1,0){40}}
\put(5,0){\vector(0,1){55}}

\put(13,46){\line(0,-1){28}}
\put(28,46){\line(0,-1){28}}

\put(28,18){\line(-1,0){15}}
\put(28,46){\line(-1,0){15}}

%\put(16,32){\linethickness{2pt}\line(1,0){13}}
\multiput(14,32)(3.4,0){4}{\linethickness{2pt}\line(1,0){2.8}}

\put(8,48){\line(0,-1){44}}
\put(33,48){\line(0,-1){44}}

\put(33,12){\line(-1,0){29}}
\put(33,48){\line(-1,0){29}}

\put(8,12){\vector(1,1){18}}
%\multiput(10,10)(5.1,5.1){4}{\line(1,1){3.6}}

\put(33,48){\vector(-1,-1){18}}
%\multiput(37,48)(-5.1,-5.1){4}{\line(-1,-1){3.6}}

\put(15,30){\linethickness{2pt}\line(1,0){11}}

\put(15,23){\circle*{1.5}}
\put(15,23){\line(1,1){4}}
\put(19,27){\circle{1.5}}
\put(15,23){\line(-1,-1){4}}
\put(11,19){\circle{1.5}}

\put(18,41){\circle*{1.5}}
\put(15,38){\line(1,1){6}}
\put(21,44){\circle{1.5}}
\put(15,38){\circle{1.5}}

\put(26,46){\circle*{1.5}}
\put(24,44){\line(1,1){4}}
\put(28,48){\circle{1.5}}
\put(24,44){\circle{1.5}}

\put(13,40){\circle*{1.5}}
\put(13,40){\line(1,1){5}}
\put(18,45){\circle{1.5}}
\put(13,40){\line(-1,-1){5}}
\put(8,35){\circle{1.5}}

\put(28,21){\circle*{1.5}}
\put(28,21){\line(1,1){5}}
\put(33,26){\circle{1.5}}
\put(28,21){\line(-1,-1){5}}
\put(23,16){\circle{1.5}}

\put(19,18){\circle*{1.5}}
\put(19,18){\line(1,1){6}}
\put(25,24){\circle{1.5}}
\put(19,18){\line(-1,-1){6}}
\put(13,12){\circle{1.5}}

\put(7,1){$q_{1}$}
\put(32,1){$p_{1}$}

%\put(45,1){$x_{1}$}

\put(-0.5,48){$p_{2}$}
\put(-0.5,12){$q_{2}$}

%\put(0,55){$x_{2}$}

\end{picture}
\end{center}
%\vspace{-2ex}
\caption{Solution to a problem with nonzero $w_{i}$.}\label{F-Chebyshev3}
%\vspace{-1ex}
\end{figure}

\section{Constrained Location Problems}

Given a feasible set $S\in\mathbb{R}^{n}$, we now consider a constrained location problem
$$
\min_{\bm{x}\in S}\max_{1\leq i\leq m}(\rho(\bm{r}_{i},\bm{x})+w_{i}).
$$

We suppose that the feasible location area $S$ is defined either by equality constraints as
$$
S_{0}
=
\left\{\bm{x}\left|\max_{1\leq j\leq n}(a_{ij}+x_{j})=x_{i},\ i=1,\ldots,n\right.\right\},
$$
or by inequality constraints as
$$
S_{1}
=
\left\{\bm{x}\left|\max_{1\leq j\leq n}(a_{ij}+x_{j})\leq x_{i},\ i=1,\ldots,n\right.\right\}.
$$

In the two-dimensional case, the constraints determine regions that are given by the intersection of half-planes with their border lines drawn at $45^{\circ}$ angle to the coordinate axes on the plane. Specifically, the intersection can take the form of a strip that can be considered as a quite natural restriction for the feasible area in the Chebyshev Messenger Boy Problem.

\subsection{Algebraic Representation and Solution}

First we represent the objective function in terms of the idempotent semifield $\mathbb{R}_{\max,+}$ to get
\begin{equation}
\min_{\bm{x}\in S}\varphi(\bm{x}),
\label{P-ChebyshevConstrained1}
\end{equation}
where
$$
\varphi(\bm{x})
=
\bm{x}^{-}\bm{p}
\oplus
\bm{q}^{-}\bm{x}.
$$

The feasible sets can be written as
$$
S_{0}
=
\{\bm{x}| A\bm{x}=\bm{x}\},
\qquad
S_{1}
=
\{\bm{x}| A\bm{x}\leq\bm{x}\},
$$

\begin{theorem}
\label{T-LPEC}
Suppose that $A$ is an irreducible matrix with $\mathop\mathrm{Tr}(A)=\mathbb{1}$, and assume that
$$
\Delta
=
\sqrt{(A^{+}(\bm{q}^{-}A^{+})^{-})^{-}\bm{p}}.
$$

Then it holds that
$$
\min_{\bm{x}\in S_{0}}\varphi(\bm{x})
=
\Delta,
$$
where the minimum is attained at the vector
$$
\bm{x}
=
\Delta A^{+}(\bm{q}^{-}A^{+})^{-}.
$$
\end{theorem}
\begin{proof}
Since $\mathop\mathrm{Tr}(A)=\mathbb{1}$ all solutions of the equation $A\bm{x}=\bm{x}$ take the form
$$
\bm{x}
=
A^{+}\bm{y}
$$
for any vector $\bm{y}\in\mathbb{R}^{m}$, where $m$ is the number of columns in the matrix $A^{+}$, $m\leq n$.

Substitution $\bm{x}=A^{+}\bm{y}$ into the objective function $\varphi(\bm{x})$ turns the problem \eqref{P-ChebyshevConstrained1} into an unconstrained problem of finding
$$
\min_{\bm{y}\in\mathbb{R}_{+}^{m}}((A^{+}\bm{y})^{-}\bm{p}\oplus\bm{q}^{-}A^{+}\bm{y}).
$$

To solve this new problem we apply Theorem~\ref{T-AxbcAx}. First we evaluate
$$
\Delta
=
\sqrt{(A^{+}(\bm{q}^{-}A^{+})^{-})^{-}\bm{p}},
$$
and then conclude that the minimum in the problem is equal to $\Delta$ and attained at $\bm{y}=\Delta(\bm{q}^{-}A^{+})^{-}$.

Going back to the constrained problem, we finally get $\bm{x}=\Delta A^{+}(\bm{q}^{-}A^{+})^{-}$.
\end{proof}

\begin{theorem}
\label{T-LPIC}
Suppose that $A$ is an irreducible matrix with $\mathop\mathrm{Tr}(A)\leq\mathbb{1}$, and assume that
$$
\Delta
=
\sqrt{(A^{\ast}(\bm{q}^{-}A^{\ast})^{-})^{-}\bm{p}}.
$$

Then it holds that
$$
\min_{\bm{x}\in S_{1}}\varphi(\bm{x})
=
\Delta,
$$
where the minimum is attained at the vector
$$
\bm{x}
=
\Delta A^{\ast}(\bm{q}^{-}A^{\ast})^{-}.
$$
\end{theorem}
\begin{proof}
It is sufficient to note that the solution of the inequality $A\bm{x}\leq\bm{x}$ is written in the form
$$
\bm{x}
=
A^{\ast}\bm{y}
$$
for any vector $\bm{y}$ of appropriate size. All further arguments are the same as in the previous theorem.
\end{proof}

\subsection{An Example}

Consider a minimax single facility location problem with $m=2$, $w_{1}=w_{2}=0$, and
$$
\bm{r}_{1}
=
\left(
\begin{array}{r}
-2
\\
5
\end{array}
\right),
\qquad
\bm{r}_{2}
=
\left(
\begin{array}{c}
6
\\
13
\end{array}
\right).
$$

Note that for this problem we have
$$
\bm{p}
=
\bm{r}_{2},
\qquad
\bm{q}
=
\bm{r}_{1}.
$$

First assume that there are no constraints for the feasible location area. Following Theorem~\ref{T-Chebyshev}, we get
$$
\Delta
=
(\bm{q}^{-}\bm{p})^{1/2}
=
4.
$$

Furthermore, we obtain
$$
\bm{x}
=
\Delta^{-1}\bm{p}
=
\Delta\bm{q}
=
\left(
\begin{array}{c}
2
\\
9
\end{array}
\right).
$$

Now consider the same problem under equality constraints
\begin{align*}
\max(x_{1},x_{2}-3)
&=
x_{1},
\\
\max(x_{1}-5,x_{2}-2)
&=
x_{2}.
\end{align*}

Representation of these constraints in terms of $\mathbb{R}_{\max,+}$ gives the vector equation
$$
A\bm{x}
=
\bm{x},
$$
where
$$
A
=
\left(
\begin{array}{rr}
0 & -3
\\
-5 & -2
\end{array}
\right).
$$

It is easy to see that the matrix $A$ is irreducible and $\mathop\mathrm{Tr}(A)=\mathbb{1}=0$. To apply Theorem~\ref{T-LPEC} we first calculate the matrices
$$
A^{\ast}
=
I\oplus A
=
\left(
\begin{array}{rr}
0 & -3
\\
-5 & 0
\end{array}
\right),
\qquad
A^{\times}
=
AA^{\ast}
=
\left(
\begin{array}{rr}
0 & -3
\\
-5 & -2
\end{array}
\right).
$$

Since only the first column in $A^{\times}$ has $\mathbb{1}=0$ on the diagonal, we take
$$
A^{+}
=
\left(
\begin{array}{r}
0
\\
-5
\end{array}
\right).
$$

Now we successively get
$$
\bm{q}^{-}A^{+}
=
2,
\qquad
A^{+}(\bm{q}^{-}A^{+})^{-}
=
\left(
\begin{array}{r}
-2
\\
-7
\end{array}
\right).
$$

Finally, we evaluate
$$
\Delta
=
\sqrt{(A^{+}(\bm{q}^{-}A^{+})^{-})^{-}\bm{p}}
=
10,
$$
and then find the solution
$$
\bm{x}
=
\Delta A^{+}(\bm{q}^{-}A^{+})^{-}
=
\left(
\begin{array}{c}
8
\\
3
\end{array}
\right).
$$

Suppose that the constraints take the form of the inequality
$$
A\bm{x}
\leq
\bm{x}.
$$

According to Theorem~\ref{T-LPIC}, we calculate
$$
\bm{q}^{-}A^{\ast}
=
\left(
\begin{array}{cc}
2 & -1
\end{array}
\right),
\qquad
A^{\ast}(\bm{q}^{-}A^{\ast})^{-}
=
\left(
\begin{array}{r}
-2
\\
1
\end{array}
\right).
$$

Then we have
$$
\Delta
=
\sqrt{(A^{\ast}(\bm{q}^{-}A^{\ast})^{-})^{-}\bm{p}}
=
6,
$$
and eventually get the solution
$$
\bm{x}
=
\Delta A^{\ast}(\bm{q}^{-}A^{\ast})^{-}
=
\left(
\begin{array}{r}
4
\\
7
\end{array}
\right).
$$

We illustrate the above solutions in Fig.~\ref{F-LPEIC}, where the given points are indicated with empty circles, whereas the location points are shown with thick dots. The solution for the unconstrained problem is located in the center of the squares, which represent isolines of the objective function. The thick dotes in bottom right vertices of squares correspond to the solution under the constraints $A\bm{x}=\bm{x}$ (left picture) and $A\bm{x}\leq\bm{x}$ (right picture) with the matrix $A=(\bm{a}_{1},\bm{a}_{2})$.
\begin{figure}[ht]
\setlength{\unitlength}{1mm}
\begin{center}
\begin{picture}(35,50)

\put(0,15){\vector(1,0){35}}
\put(10,0){\vector(0,1){50}}

\put(10,15){\thicklines\vector(-3,-2){6}}

\put(3,10){\line(1,1){25}}

\put(10,15){\thicklines\vector(0,-1){10}}
\put(7,2){\thicklines\line(1,1){24}}

\put(26,21){\line(-1,0){24}}
\put(26,21){\line(0,1){24}}
\put(26,45){\line(-1,0){24}}
\put(2,21){\line(0,1){24}}

\put(22,25){\line(-1,0){16}}
\put(22,25){\line(0,1){16}}
\put(22,41){\line(-1,0){16}}
\put(6,25){\line(0,1){16}}

\put(6,25){\circle{1.5}}

\put(22,41){\circle{1.5}}

\put(14,33){\circle*{1.5}}

%\put(18,29){\circle*{1.5}}

\put(26,21){\circle*{1.5}}

\put(4,8){$\bm{a}_{2}$}

\put(12,3){$\bm{a}_{1}$}

\end{picture}
\hspace{20\unitlength}
\begin{picture}(35,50)

\put(0,15){\vector(1,0){35}}
\put(10,0){\vector(0,1){50}}

\put(10,15){\thicklines\vector(-3,-2){6}}
\put(2,13){\thicklines\line(1,1){25}}
\multiput(3,14)(1,1){24}{\line(1,0){1}}

\put(3,10){\line(1,1){25}}

\put(10,15){\thicklines\vector(0,-1){10}}
\put(7,2){\thicklines\line(1,1){24}}
\multiput(8,3)(1,1){23}{\line(-1,0){1}}

\put(18,29){\line(-1,0){8}}
\put(18,29){\line(0,1){8}}
\put(10,29){\line(0,1){8}}
\put(10,37){\line(1,0){8}}

\put(22,25){\line(-1,0){16}}
\put(22,25){\line(0,1){16}}
\put(22,41){\line(-1,0){16}}
\put(6,25){\line(0,1){16}}

\put(6,25){\circle{1.5}}

\put(22,41){\circle{1.5}}

\put(14,33){\circle*{1.5}}

\put(18,29){\circle*{1.5}}

%\put(26,21){\circle*{1.5}}

\put(4,8){$\bm{a}_{2}$}

\put(12,3){$\bm{a}_{1}$}

\end{picture}
\end{center}
%\vspace{-2ex}
\caption{Solutions to a location problem under equality constraints (left) and inequality constraints (right).}\label{F-LPEIC}
%\vspace{-1ex}
\end{figure}
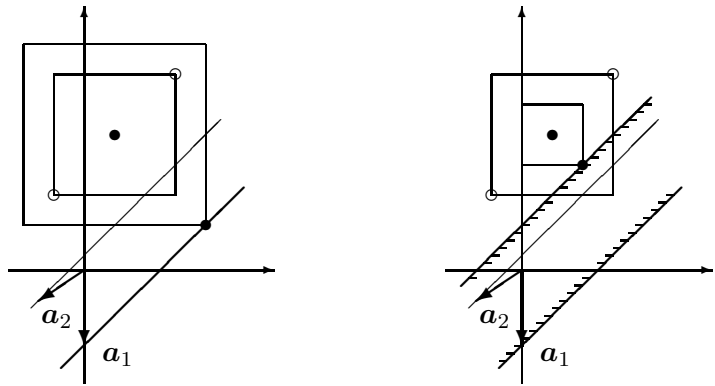

\subsection*{Acknowledgments}
The author is very grateful to an anonymous reviewer for valuable comments and constructive suggestions. In particular, he thanks the reviewer for pointing out inequality \eqref{I-DbAxDc} that provided a useful way to weaken assumptions on the underlying semifield as well as to simplify related proofs and presentation of results.

\bibliography{KrivulinN}

\end{document}